\documentclass[a4paper,12pt]{amsart}
\usepackage{mathtools}
\usepackage{amsfonts,amsthm,amsmath,amssymb}
\usepackage{array}
\usepackage{epsfig}
\usepackage{ragged2e}
\usepackage{graphics}
\usepackage{fancyhdr}
\usepackage{ifthen}
\usepackage{xcolor}
\usepackage[pdftex,bookmarks=true]{hyperref}
\nonstopmode \numberwithin{equation}{section}
\makeatletter
\newcommand{\ubrace}[2]{\mathord{\mathpalette\ubrace@{{#1}{#2}}}}
\newcommand{\ubrace@}[2]{\ubrace@@#1#2}
\newcommand{\ubrace@@}[3]{
	\underbrace{#1#2}_{#3}%
}
\makeatother

\newtheorem{Th}{Theorem}[subsection]

\theoremstyle{definition}
\newtheorem{defn}[Th]{Definition}

\newtheorem{Thm}{Theorem}[section]
\newtheorem{Cor}[Thm]{Corollary}
\newtheorem{Lem}[Thm]{Lemma}
\newtheorem{Prop}[Thm]{Proposition}

\theoremstyle{definition}

\newtheorem{Rem}[Thm]{Remark}
\newtheorem{Ex}[Thm]{Example}
\begin{document}
	\bibliographystyle{amsplain}
	\title{A Combinatorial Criterion and Center for the quasi-isometry groups of Euclidean spaces }
	\author{Swarup Bhowmik}
	\address{Swarup Bhowmik, Department of Mathematics,
		Indian Institute of Technology Kharagpur, Kharagpur - 721302, India.}
	\email{swarup.bhowmik@iitkgp.ac.in}
	\author{Prateep Chakraborty}
	\address{Prateep Chakraborty, Department of Mathematics,
		Indian Institute of Technology Kharagpur, Kharagpur - 721302, India.}
	\email{prateep@maths.iitkgp.ac.in}

	\begin{abstract}
		In this study, we introduce the notion of $PL_\delta$-homeomorphisms of $\mathbb{R}^n$. Furthermore, we provide a combinatorial criterion reliant on the vertices and edges of simplicial structures, to determine whether a piecewise-linear homeomorphism to be a quasi-isometry. By employing this criterion, we subsequently show that the center of the group $QI(\mathbb{R}^n)$, which comprises all quasi-isometries of $\mathbb{R}^n$, is indeed trivial.
	\end{abstract}
	\thanks{The first author of this article acknowledges the financial support from Inspire, DST, Govt. of India as a Ph.D. student (Inspire) sanction
		letter number: No. DST/INSPIRE Fellowship/2018/IF180972}
	\maketitle
	\pagestyle{myheadings}
	\markboth{S. Bhowmik and P. Chakraborty}{A Combinatorial Criterion and Center for the quasi-isometry groups of Euclidean spaces }
	\bigskip
	\section{Introduction}
	Consider a metric space denoted as $X$. The group comprising quasi-isometries from $X$ to itself is formally represented as $QI(X)$ and stands as a quasi-isometric invariant of $X$. Several studies have explicitly examined $QI(X)$ within a specific categories of metric spaces and finitely generated groups with word metric. Notable examples include investigations involving irreducible lattices in semisimple Lie groups \cite{Farb}, solvable Baumslag-Solitar groups $BS(1, n)$ \cite[Theorem 7.1]{Farb Mosher}, as well as the groups $BS(m, n)$ where $1 < m < n$ \cite[Theorem 4.3]{Whyte}.
	
	The exploration of the group of quasi-isometries of Euclidean spaces remains relatively limited. Nevertheless, Mitra and Sankaran \cite{Mitra Sankaran} made notable observations, revealing that $QI(\mathbb{R}^n)$ constitutes a sizable group, containing various recognized subgroups of the homeomorphism group of $\mathbb{S}^{n-1}$, for example, $Bilip(\mathbb{S}^{n-1})$, $Diff^{r}(\mathbb{S}^{n-1})$, $PL(\mathbb{S}^{n-1})$ and others like $Diff^{r}(\mathbb{D}^{n}, \mathbb{S}^{n-1})$ and $Diff^{r}_{k}(\mathbb{R}^{n})$. Notably, Gromov and Pansu \cite[\textsection 3.3.B]{Gromov} demonstrated that for $n = 1$, $QI(\mathbb{Z}) (\cong QI(\mathbb{R})$) emerges as an infinite-dimensional group. It has been established that $QI(\mathbb{R})$ contains the free group with a rank equal to the continuum and a copy of the group comprising all compactly supported diffeomorphisms of $\mathbb{R}$ \cite{Sankaran}, with its center being trivial  \cite{Chakraborty}. However, additional research is required to unveil the full extent of the structure and behavior of $QI(\mathbb{R}^n)$.
	
	In \cite{Sankaran}, an explicit description of any quasi-isometries of the real line was presented, employing piecewise-linear homeomorphisms with specific conditions on their slopes. In this research article, we extend this study by providing combinatorial criterion that determine when a piecewise-linear homeomorphism of $\mathbb{R}^n$ qualifies as a quasi-isometry. To achieve this, we establish conditions related to vertices and edges of a triangulation of $\mathbb{R}^n$ that are sufficient to determine whether a piecewise-linear homeomorphism of $\mathbb{R}^n$ is a quasi-isometry.

	In Section 2, we introduce the definition of a piecewise-linear homeomorphism considering a general simplicial complex embedded within a Euclidean space with simplices satisfy the condition of convexity. In this paper, we shall denote the Euclidean metric as $d$ consistently throughout the discussions.
	
	Now, we consider two equivalent locally finite simplicial complexes (l.f.s.c. in abbreviation) $\mathcal{K}_1$ and $\mathcal{K}_2$ of $\mathbb{R}^{n}$ (the definitions are given in \textsection\ref{PL-homeomorphisms}). We denote the angles  by $\theta_{ijk}^{\Delta}$ between two adjacent edges $v_i,v_j$ and $v_k$ of a simplex $\Delta\in\mathcal{K}_1$ after fixing an ordering of the vertices of $\Delta$ for $0\leq i<k\leq n,~j\neq i,~j\neq k$. We define the set $A=\left\{(\cos\theta_{ijk}^{\Delta})_{i,j}\subset\displaystyle\prod_{\substack{0\leq i<k\leq n\\j\neq i,k}}(-1,1):~\Delta\in\mathcal{K}_1\right\}$ for the l.f.s.c. $\mathcal{K}_1$. We similarly define the set $B$ for the l.f.s.c. $\mathcal{K}_2.$ The sets $A$ and $B$ contain all the informations about the angles between two adjacent edges of the simplices of $\mathcal{K}_1$ and $\mathcal{K}_2$. The next result gives a sufficient condition for a simplicial isomorphism to be a quasi-isometry in terms of the distance of the vertices of any simplex and the sets $A$ and $B$.
	\begin{Thm}\label{1.1}
		\textit{Let $g\in PL(\mathbb{R}^n)$ and $g=|f|$ where $f:\mathcal{K}_1\rightarrow\mathcal{K}_2$ ($\mathcal{K}_1$ and $\mathcal{K}_2$ are equivalent simplicial structures of $\mathbb{R}^n$) is a simplicial isomorphism. If
		\begin{enumerate}
		\item [(i)] there exists $\mu>1$ such that $\displaystyle\frac{1}{\mu}
		\leq\displaystyle\frac{d(f(v_i),f(v_j))}{d(v_i,v_j)}
		\leq\mu$, where $v_i,v_j$ are any pair of vertices of any simplex of  $\mathcal{K}_1$ and $i\neq j$ and
		\item [(ii)] $A$ and $B$ are closed subsets of $\left\{\displaystyle\prod_{\substack{0\leq i<k\leq n\\j\neq i,k}}\mathbb{R}\right\}$ for some ordering of the vertices of simplices of $\mathcal{K}_1$ and $\mathcal{K}_2$,
		\end{enumerate}
		then $g$ is a quasi-isometry of $\mathbb{R}^n$ and $g$ is quasi-isometrically equivalent to the identity map if and only if $d(f(v_i),v_i)<C$ for some $C>0$ and for all $v_i\in \mathcal{K}_1$.}
	\end{Thm}
	Exploring the algebraic properties of $QI(\mathbb{R}^n)$ represents an alternative direction of investigation. A fundamental preliminary aspect in this pursuit is to identify the center of  $QI(\mathbb{R}^n)$.
	\begin{Thm}\label{1.2}
	  \textit{The center of the group $QI(\mathbb{R}^n)$ is trivial.}
	\end{Thm}
	 In Section 2, we introduce the notions for two essential components: $PL(X)$, $PL_\delta(X)$ and we define the length metric for an l.f.s.c. space $X$ that is embedded in an $n$-dimensional Euclidean space $\mathbb{R}^n$. In Section 3, we establish a crucial connection between $PL_\delta(X)$, $Bilip(X)$, and $QI(X)$, which yields a generalized version of Theorem \ref{1.1} in the form of Corollary \ref{3.6}. Notably, Corollary \ref{3.5} and Corollary \ref{3.6} demonstrate a link between combinatorics, topology, and geometry. In Section 4, we use Corollary \ref{3.6} as a tool to prove Theorem \ref{1.2}.
	\section{Preliminaries}
	In this section, at first we describe the concept of quasi-isometry. Then we discuss about the l.f.s.c. and its equivalence. Subsequently, we define the set of piecewise-linear homeomorphisms with respect to a fixed l.f.s.c., which turns into a group due to the equivalence of the structure. Then by introducing $PL(X)$, we have captured the entirety of possible piecewise-linear homeomorphisms for a given space $X$. In parallel, we introduce $PL_\delta(X).$

	We begin by recalling the notion of quasi-isometry. Let $f:(X,d_1)\rightarrow (X',d_2)$ be a map (which is not assumed to be continuous) between two metric spaces. We say that $f$ is a quasi-isometric embedding if there exists an $M>1$ such that
	\begin{center}
		$\displaystyle\frac {1}{M}d_1(x_1,x_2)-M\leq d_2(f(x_1),f(x_2))\leq Md_1(x_1,x_2)+M$
	\end{center}
	for all $x_1,x_2\in X$. In addition to this, if there exists $M'>0$ such that for any $y\in X',$ there is an $x\in X$ such that $d_2(f(x),y)<M',$ then $f$ is said to be a quasi-isometry. If $f$ is a quasi-isometry, then there exists a quasi-isometry $f':X'\rightarrow X$ such that $f'\circ f$ (resp. $f\circ f'$) is quasi-isometrically equivalent to the identity map of $X$ (resp. $X'$). (Two maps $f,g:X\rightarrow X$ are said to be quasi-isometrically equivalent if there exists a constant $\alpha>0$ such that $d_1(f(x),g(x))<\alpha$, for all $x\in X$ ). Let [$f$] denote the equivalence class of a quasi-isometry $f:X\rightarrow X$. We denote the set of all equivalence classes of quasi-isometries of $X$ by ${QI(X)}$. This set forms a group under the composition [$f$]$\circ$[$g$]=[$f\circ g$] for $[f],[g]\in QI(X)$. If $f:X\rightarrow X'$ is a quasi-isometry, then $QI(X)$ is isomorphic to $QI(X')$ as a group.
	
	\vspace*{2mm}
	Recall that a homeomorphism $f:(X,d)\rightarrow (X,d)$ is \textit{bi-Lipschitz} if there exists
	a constant $\lambda\geq 1$ such that $\displaystyle\frac {1}{\lambda}d(x, y)\leq d(f(x),f(y))\leq \lambda d(x,y).$ We denote the set of all bi-Lipschitz homeomorphisms from $X$ to itself as  $Bilip(X)$. It is evident that if $f\in  Bilip(X)$, then $f$ is a quasi-isometry of $X$.
	
	\vspace*{2mm}
	Now, we proceed to introduce the notion of $PL_\delta$- homeomorphisms of a path-connected locally finite simplicial complex $X\subset\mathbb{R}^n$. This notion serves as a generalization of $PL_\delta(\mathbb{R})$ defined in \cite{Sankaran}. We refer the reader to \cite{Hudson} for basic concepts of a simplicial complex.
	\subsection{\textit{PL-homeomorphisms of $X$}}\label{PL-homeomorphisms}
	\begin{defn}
		A $r$-simplex in $\mathbb{R}^{n}$ is the convex hull of $(r+1)$ linearly independent points called its vertices. Each face of an $r$-simplex is the convex span of some of the vertices and therefore is an $m$-simplex where $m\leq r$. For two simplices $\sigma$ and $\tau$, we write $\sigma<\tau$ if $\sigma$ is a face of $\tau$. An l.f.s.c. $\mathcal{T}$ in $\mathbb{R}^{n}$ is a collection of simplices such that \\
		(i) $\sigma,\tau\in \mathcal{T}$ implies either $\sigma\cap\tau=\phi$ or $\sigma\cap\tau$ is a common face of both $\sigma$ and $\tau$,\\
		(ii) $\tau\in\mathcal{T}$ and $\sigma<\tau$ implies $\sigma$ is a simplex in $\mathcal{T}$ and \\
		(iii) For all $x\in |\mathcal{T}|$, there exists a neighbourhood $U$ of $x$ in $\mathbb{R}^{n}$ which interesects only finitely many simplices of $\mathcal{T}$.
	\end{defn}
	\begin{defn}\label{locally finite}
		Let $X\subset \mathbb{R}^n$ be a path-connected space. Let $\mathcal{T}$ be an l.f.s.c. such that $|\mathcal{T}|=X$ \cite[Chapter I]{Hudson}. Two such l.f.s.c. $\mathcal{T}_1$ and $\mathcal{T}_2$ are said to be \textit{equivalent} if there are subdivisions $\mathcal{T}_1^{'}$ and $\mathcal{T}_2^{'}$ of $\mathcal{T}_1$ and $\mathcal{T}_2$ respectively such that there is a simplicial isomorphism $\psi:\mathcal{T}_1^{'}\rightarrow \mathcal{T}_2^{'}$ and $\psi$ realizes identity of $X$ . It can be readily shown that this relation among the collection of l.f.s.c. of $X$ is an equivalence relation \cite[Theorem 3.6]{Hudson} and we denote the collection of equivalence classes as $\mathbb{T}$. A space $X$ is said to be an l.f.s.c. space if there exists an l.f.s.c. $K$ such that $|K|=X$.
	\end{defn}
	\begin{defn}\label{1.1.2}
		Let us fix an l.f.s.c. $\mathcal{K}_0$ of $X$. We say that a homeomorphism $g$ of $X$ is \textit{piecewise linear with respect to $\mathcal{K}_0$} if there exist two l.f.s.c. $\mathcal{K}_1$ and $\mathcal{K}_2$ of $X$ which are equivalent to $\mathcal{K}_0$ and a simplicial isomorphism $f:(X;\mathcal{K}_1)\rightarrow (X;\mathcal{K}_2)$ that realises $g$, that is $|f|=g$. We denote the set of all such piecewise-linear homeomorphisms with respect to $\mathcal{K}_0$ by $PL(X;\mathcal{K}_0)$.
	\end{defn}
	It can be shown that $PL(X;\mathcal{K}_0)$ is a subgroup of $Homeo(X)$ and if $\mathcal{K}_0$ and $\mathcal{K}_0^{'}$ are equivalent triangulations of $X$, then $PL(X;\mathcal{K}_0)$ and $PL(X;\mathcal{K}_0^{'})$ are the same subsets of $Homeo(X)$. There is a homomorphism from the free product  $\star\displaystyle_{[\mathcal{K}]\in \mathbb{T}}PL(X;\mathcal{K})$ to $Homeo(X)$. We denote the image of this homomorphism as $PL(X)$ in $Homeo(X)$.
	\subsection{$PL_\delta$-homeomorphisms of $X$}
	\begin{defn}\label{PL_delta}
	We say that a homeomorphism $g\in PL(X;\mathcal{K}_0)$ is a \textit{$PL_\delta$-homeomo-rphism with respect to $\mathcal{K}_0$} if there exists a simplicial isomorphism $f:\mathcal{K}_1\rightarrow \mathcal{K}_2$, ($\mathcal{K}_1,\mathcal{K}_2$ are equivalent to $\mathcal{K}_0),~  |g|=f$, and a constant $M>1$ such that for each simplex $\Delta\in \mathcal{K}_1$, 
	$\displaystyle\frac{1}{M}\leq\displaystyle\frac{d(g(x),g(y))}{d(x,y)}\leq M$, for all $x,y\in\Delta$ and $x\neq y$.
	\end{defn}
    This can be shown that $PL_\delta(X;\mathcal{K}_0)$ is a subgroup of $PL(X;\mathcal{K}_0)$ and there is a homomorphism from the free product  $\star\displaystyle_{[\mathcal{K}]\in \mathbb{T}}PL_\delta(X;\mathcal{K})$ to $Homeo(X)$. We denote the image of this homomorphism as $PL_\delta(X)$ in $Homeo(X)$.\\
    
    Now, we need to introduce a metric for an l.f.s.c. space $X$. In an l.f.s.c. space, the property of being path-connected implies any two of its vertices are connected by edges \cite[Chapter 3]{Spanier}. Therefore any two points of $X$ are joined by union of finitely many straight lines $C_1,C_2,\cdots,C_k$ where the end point of $C_i$ is the begining point of $C_{i+1}$, $i=1,2,\cdots,k-1.$
    For a straight line $S$, we define the length $l(S)$ as the distance  between its starting and ending points. Now, we define the length or intrinsic metric $d_p$ on $X$ as follows:\\
    $d_p(x,y)=inf\Big\{\displaystyle\sum_{i=1}^{k}l(C_i)$: $C$ is a path from $x$ to $y$, consisting of finitely many straight lines $C_1,C_2,\ldots,C_k$ for which $C_i(1)=C_{i+1}(0)$, for $i=1,2,\ldots,k-1\Big\}$.\\
    The length metric $d_p$ is said to be equivalent to the Euclidean metric $d$ if there exist two constants $\lambda_1,\lambda_2>0$ such that $\lambda_1d(x,y)\leq d_p(x,y)\leq \lambda_2d(x,y)$, for all $x,y\in X$. Note that for a convex set $X$, both the length metric $d_p$ and the Euclidean metric $d$ are equal.
	\section{Proof of Theorem \ref{1.1}}
	This section presents a combinatorial criterion, specified in Corollary \ref{3.6}, which establishes the conditions under which a piecewise-linear homeomorphism $g$ of a path-connected l.f.s.c. space $X$ becomes a quasi-isometry. Moreover, we show that the first conclusion of Theorem \ref{1.1} directly follows as a consequence of Corollary \ref{3.6}. To prove Corollary 3.6, we establish the connections between $PL_\delta(X)$, $Bilip(X)$ and $QI(X)$.
	\begin{Lem}\label{3.1}
		\textit{Let $X$ be a path-connected l.f.s.c. space with the length metric $d_p$ equivalent to the Euclidean metric $d$. Then  $PL_\delta(X)$ is a subgroup of $Bilip(X)$.}
	\end{Lem}
	\begin{proof}
		It is enough to show that $PL_\delta(X;\mathcal{K}_0)$ is a subgroup of $Bilip(X)$ for each equivalence class $[\mathcal{K}_0]$ in $\mathbb{T}$ by using the universal property of free product of group.\\
		We fix $g\in PL_\delta(X;\mathcal{K}_0)$. Then there exist two l.f.s.c. $\mathcal{K}_1$ and $\mathcal{K}_2$, both equivalent to $\mathcal{K}_0$, and a simplicial isomorphism $f:\mathcal{K}_1\rightarrow \mathcal{K}_2$ such that $|f|=g$ and a constant $M>1$ such that for each simplex $\Delta\in \mathcal{K}_1$, 
		$\displaystyle\frac{1}{M}\leq\displaystyle\frac{d(g(x),g(y))}{d(x,y)}\leq M$, for all $x,y\in\Delta$ and $x\neq y$. Let $x,y\in X$. Then two cases arise:\\
		\textbf{Case I:} There exists a simplex $\Delta\in \mathcal{K}_1$ such that $x,y\in \Delta$. Then for all $x,y\in\Delta,$
		\begin{equation}{\label{p1 eq1}}
		\frac {1}{M}d(x,y)
		\leq d(g(x),g(y))\leq Md(x,y).
		\end{equation}
		\textbf{Case II:} There are two different simplices $\Delta,~\Delta^{'}\in \mathcal{K}_1$ such that $x\in \Delta $, $y\in \Delta^{'}$. \\
	    Let $C$ be a finite piecewise straight line from $x$ to $y$ such that $C=\displaystyle\bigcup_{i=1}^{m}C_i$, where each $C_i$ is a straight line for which $C_1(0)=x$ and $C_m(1)=y$ and $d_p(x,y)\leq \displaystyle\sum_{i=1}^{m}l(C_i)<d_p(x,y)+\epsilon$, for a given $\epsilon>0$.\\
	    We claim that for each $i$, $C_i$ intersects only finitely many simplices of $\mathcal{K}_1$. Suppose, if possible, for some $i$, $C_i$ intersects infinitely many simplices of $\mathcal{K}_1$, then we can form an infinite set $A\subset C_i$ with points from infinitely many mutually distinct simplices. The set $A$ has a limit point, say $\alpha$. Then any neighbourhood of $\alpha$ intersects infinitely many simplices of $\mathcal{K}_1$ which violates locally finiteness of $\mathcal{K}_1$. Thus, our claim is established. Then for each $i=1,2,\cdots,m,$ we can enumerate these simplices as $\Delta_1^{i},\Delta_2^{i},\cdots ,\Delta_{r_i}^{i}$ and find points $a_0^{i},a_1^{i},\cdots,a_{r_i}^{i}\in[0,1]$ such that 
		\\ (i) $a_j^{i}\in C_i^{-1}(\Delta_j^{i})\cap C_i^{-1}(\Delta_{j+1}^{i})$ for all $j=1,2,3,\cdots,r_{i}-1$ and \\(ii) $[a_{j-1}^{i},a_{j}^{i}]\subset C_i^{-1}(\Delta_{j}^{i})$ for all $j=1,2,3,\cdots,r_i$.\\ Then $g\big([a_j^{i},a_{j+1}^{i}]\big)$ is a straight line. Thus,
		\begin{align*}
			d(g(x),g(y))&\leq \sum_{i=1}^{m}\sum_{j=0}^{r_{i}-1}l\big[g(a_j^{i}),g(a_{j+1}^{i})\big]\\
			&\leq M\sum_{i=1}^{m}\sum_{j=0}^{r_{i}-1}l[a_j^{i},a_{j+1}^{i}]\\&=M\Big(\sum_{i=1}^{m}l(C_i)\Big)\\
			&\leq M(d_p(x,y)+\epsilon)\\&\leq M(\lambda_2d(x,y)+\epsilon).
		\end{align*}
	Then from above we get,
	\begin{equation}\label{p1 eq2}
		d(g(x),g(y))\leq M'd(x,y), ~\text{where ~}M'=M\lambda_2.
	\end{equation}
		Again, since $g\in PL_\delta(X;\mathcal{K}_0)$, $g^{-1}$ also belongs to $PL_\delta(X;\mathcal{K}_0)$, so
		\begin{equation}\label{p1 eq3}
			{d(x,y)}\leq M'{d(g(x),g(y))}.
		\end{equation}	
		Combining (\ref{p1 eq1}), (\ref{p1 eq2}) and (\ref{p1 eq3}), we observe that, $g\in Bilip(X)$.
	\end{proof}
	From the fact that a bi-Lipschitz homeomorphism of $X$ is a quasi-isometry of $X$, we can deduce the following corollary:
	\begin{Cor}\label{3.2}
	\textit{Let $X$ be a path-connected l.f.s.c. space with the length metric $d_p$ equivalent to the Euclidean metric $d$. If $g$ is an element of $PL_\delta(X)$, then $g$ is a quasi-isometry of $X$.}
	\end{Cor}
	Subsequently, we present a combinatorial criterion to determine whether an element $g$ of $PL(X)$ belongs to $PL_\delta(X)$. To achieve this, we initially establish a criterion when considering simplicial isomorphisms between two simplices. This particular analysis serves as a tool for proving the result concerning path-connected l.f.s.c. space $X$.
	\begin{Prop}\label{3.3}
		Let $T:\mathbb{R}^{n}\rightarrow\mathbb{R}^{n}$ be a linear isomorphism with a basis $\{e_1,e_2,\\ \cdots,e_n\}$ of $\mathbb{R}^{n}$ and $||e_i||=1$, for all $i=1,2,\cdots,n.$ If there exists $\mu>1$ such that $\displaystyle\frac {1}{\mu}\leq ||T(e_i)||\leq \mu$, then $\displaystyle\frac {1}{M'}\leq \displaystyle\frac {||T(v_1)-T(v_2)||}{||v_1-v_2||}\leq M$ for some constants $M,M'>0$ and for all $v_1,v_2\in\mathbb{R}^{n}$, $v_1\neq v_2$.
	\end{Prop}
\begin{proof}
	Let us consider the continuous function $F:\mathbb{S}^{n-1}\rightarrow\mathbb{R}$ defined by 
	\begin{equation*}
		(x_1,x_2,\cdots,x_n)\mapsto \sum_{i=1}^{n}x_i^2||e_i||^2+2\sum_{1\leq i<j\leq n}x_ix_j<e_i,e_j>=||\sum_{i=1}^{n}x_ie_i||^{2}.
	\end{equation*}
Since $\{e_1,e_2,\cdots,e_n\}$ is a basis, $\displaystyle\sum_{i=1}^{n}x_ie_i$ is always a non-zero vector in $\mathbb{R}^{n}$, so $F$ always assumes positive real numbers. Again since $\mathbb{S}^{n-1}$ is compact, $F$ admits its minimum $M_1$ (for some constant $M_1>0$).\\
Now, if $v=\displaystyle\sum_{i=1}^{n}x_ie_i\in\mathbb{R}^{n}\setminus \{0\}$, then 
\begin{align*}
	&\Big(\displaystyle\frac {x_1}{(\sum_{i=1}^{n}x_i^{2})^{\frac {1}{2}}},\displaystyle\frac {x_2}{(\sum_{i=1}^{n}x_i^{2})^{\frac {1}{2}}},\cdots,\displaystyle\frac {x_n}{(\sum_{i=1}^{n}x_i^{2})^{\frac {1}{2}}}\Big)\in \mathbb{S}^{n-1}\\
	&\implies F\Big(\displaystyle\frac {x_1}{(\sum_{i=1}^{n}x_i^{2})^{\frac {1}{2}}},\displaystyle\frac {x_2}{(\sum_{i=1}^{n}x_i^{2})^{\frac {1}{2}}},\cdots,\displaystyle\frac {x_n}{(\sum_{i=1}^{n}x_i^{2})^{\frac {1}{2}}}\Big)\geq M_1\\&\implies \displaystyle\sum_{i=1}^{m}\Big( \displaystyle\frac {x_1}{(\sum_{i=1}^{n}x_i^{2})^{\frac {1}{2}}}\Big)^2||e_i||^{2}+2\sum_{1\leq i< j\leq n}\frac {x_i}{(\sum_{i=1}^{n} x_i^{2})^{\frac {1}{2}}}\frac {x_j}{(\sum_{i=1}^{n} x_i^{2})^{\frac {1}{2}}}<e_i,e_j>~\geq M_1\\&\implies \displaystyle\sum_{i=1}^{n}x_i^2||e_i||^{2}+2\displaystyle\sum_{1\leq i<j\leq n}x_ix_j<e_i,e_j>~\geq M_1\big(\displaystyle\sum_{i=1}^{n}x_i^{2}\big).
\end{align*}
Hence, for all $v\in \mathbb{R}^{n}\setminus\{0\}$, we have $||v||^{2}\geq M_1\Big(\displaystyle \sum_{i=1}^{n}x_i^{2}\Big)$.\\
Let $v_1,v_2\in\mathbb{R}^{n}$ such that $v_1\neq v_2$ and we wrtie $v_1-v_2=\displaystyle\sum_{i=1}^{n}x_ie_i$. Then,
\begin{align*}
	||T(v_1)-T(v_2)||^{2}&=||\sum_{i=1}^{n}x_iT(e_i)||^{2}\\&\leq \big( \sum_{i=1}^{n}|x_i|^2\big)\Big(\sum_{i=1}^{n}||T(e_i)||^{2}\Big)\\&\leq n^2\mu^{2}\frac {1}{M_1}||v_1-v_2||^2.
\end{align*}
Thus we can write $\displaystyle\frac {||T(v_1)-T(v_2)||}{||v_1-v_2||}\leq M$, where $M=\frac{\mu n}{\sqrt{M_1}}$.\\
Similarly we can deduce the other inequality. \\
This completes the proof.
\end{proof}
\begin{Rem}\label{3.4} Note that in the Prop \ref{3.3}, the constant $M_1$ and hence $M$ can be chosen to be continuous functions of $<e_i,e_j>$ for $1\leq i<j\leq n$ and hence continuous function of $A$.
\end{Rem}
Regarding the definition of $A$, it can be noted that if a maximal simplex is of dimension less tha $n$, we can still include the ordered set of cosines of angles of adjacent edges into $A$ by putting 0's in the remaining components. Prop. \ref{3.3} and Remark \ref{3.4} are still valid in this defintion. The set $B$ can be adjusted similarly. From Prop. \ref{3.3}, the ensuing result establishes a sufficient condition that determines whether a piecewise-linear homeomorphism of $X$ belongs to $PL_\delta(X)$.
\begin{Cor}\label{3.5}
	\textit{Let $X$ be an l.f.s.c. space and $g\in PL(X;\mathcal{K}_{0})$ for which $g=|f|$ where $f:\mathcal{K}_1\rightarrow\mathcal{K}_2$ ($\mathcal{K}_1$ and $\mathcal{K}_2$ are equivalent l.f.s.c. to $\mathcal{K}_{0}$) is a simplicial isomorphism. If
	\begin{enumerate}
		\item [(i)] there exists $\mu>1$ such that $\displaystyle\frac{1}{\mu}\leq\displaystyle\frac{d(g(v_i),g(v_j))}{d(v_i,v_j)}\leq \mu$, where $v_i,v_j$ are any pair of vertices of any simplex of $\mathcal{K}_1$ and $i\neq j$ and
		\item [(ii)] $A$ and $B$ are closed subsets of $\left\{\displaystyle\prod_{\substack{0\leq i<k\leq n\\j\neq i,k}}\mathbb{R}\right\}$ for some ordering of the vertices of simplices of $\mathcal{K}_1$ and $\mathcal{K}_2$,
	\end{enumerate}
then $g\in PL_\delta(X)$.}
\end{Cor}
\begin{proof}
	Let $\Delta^{n}\in\mathcal{K}_1$ and $f(\Delta^{n})\in\mathcal{K}_2$. We denote the vertices of $\Delta^{n}$ by $\{v_0,v_1,\cdots\\,v_n\}$ and so $\{f(v_0),f(v_1),\cdots,f(v_n)\}$ is the set of vertices of $f(\Delta^{n})$. Without loss of generality, we may assume $v_0$ and $f(v_0)$ are origin of $\mathbb{R}^n$. Then, $g|_{\Delta^n}$ is $T|_{\Delta^n}$ where $T:\mathbb{R}^n\rightarrow\mathbb{R}^n$ is a linear isomorphism defined by $T(v_i)=g(v_i)$ for $i=1,2,\cdots,n$. We may assume $e_i=\displaystyle\frac {v_i}{||v_i||}$, $i=1,2,\cdots,n$. Then the condition of Prop \ref{3.3} is satisfied due to given condition (i).\\
	Therefore, $\displaystyle\frac {1}{M_2}\leq \displaystyle\frac {d(g(x),g(y))}{d(x,y)}\leq M_1$, where $x\neq y$ in $\Delta^n$ and $M_1,M_2$ are continuous functions on $A$ and $B$ respectively.\\
	Since $A$ and $B$ are closed subsets of $\left\{\displaystyle\prod_{\substack{0\leq i<k\leq n\\j\neq i,k}}\mathbb{R}\right\}$, so is compact, so they have positive minimum, which gives $M>1$ such that $\displaystyle\frac {1}{M}\leq \displaystyle\frac {d(g(x),g(y))}{d(x,y)}\leq M$, for all $x,y\in\Delta^n$ and for all $\Delta^n\in\mathcal{K}_1$, with $x\neq y$.
\end{proof}
Now, combining above corollary with Corollary \ref{3.2}, we obtain the following sufficient combinatorial criterion.
\begin{Cor}\label{3.6}
	\textit{Let $X$ be a path-connected l.f.s.c. space with the length metric $d_p$ equivalent to the Euclidean metric $d$. Let $g\in PL(X;\mathcal{K}_{0})$ and $g=|f|$ where $f:\mathcal{K}_1\rightarrow\mathcal{K}_2$ ($\mathcal{K}_1$ and $\mathcal{K}_2$ are equivalent to $\mathcal{K}_{0}$) is a simplicial isomorphism. If
	\begin{enumerate}
		\item [(i)] there exists $\mu>1$ such that $\displaystyle\frac{1}{\mu}\leq\displaystyle\frac{d(g(v_i),g(v_j))}{d(v_i,v_j)}\leq \mu$, where $v_i,v_j$ are any pair of vertices of any simplex of $\mathcal{K}_1$ and $i\neq j$ and
		\item [(ii)] $A$ and $B$ are closed subsets of $\left\{\displaystyle\prod_{\substack{0\leq i<k\leq n\\j\neq i,k}}\mathbb{R}\right\}$ for some ordering of the vertices of simplices of $\mathcal{K}_1$ and $\mathcal{K}_2$,
	\end{enumerate}
then $g$ is a quasi-isometry of $X$.}
\end{Cor}
    We are now ready to prove Theorem \ref{1.1}.
	\subsection{Proof of Theorem 1.1}
	As $\mathbb{R}^n$ is a convex space, the length metric $d_p$ is equal to the Euclidean metric $d$. Also note that $\mathbb{R}^n$ is locally compact, so any simplicial complex structures on $\mathbb{R}^n$ is locally finite. Consequently, the first conclusion directly follows from the preceding corollary. To address the second conclusion, we claim a more comprehensive and generalized statement.\\
	Claim: For an l.f.s.c. space $X$ and $\mathcal{K}_0\in \mathbb{T}$, the set 
	$Ker\{PL_\delta(X;\mathcal{K}_0)\rightarrow QI(X)\}$ is equal to the set $\big\{g=|f|~\text{for~some~}f:\mathcal{K}_1\rightarrow\mathcal{K}_2 ~(\mathcal{K}_1 ~\text{and}~ \mathcal{K}_2 ~\text{and}~ \text{equivalent~to~}\mathcal{K}_0):d(g(v_i),v_i)<C,~\text{for~some}~C>0~\text{and ~for~all~vertices~}v_i\in\mathcal{K}_1\big\}$.\\
	Note that $g\in Ker\{PL_\delta(X;\mathcal{K}_0)\rightarrow QI(X)\}$ implies that $d(g(x),x)<C$, for some constant $C>0$ and for all $x\in X$.\\
	For the converse part, let $x$ be a point in the interior of a simplex $\{v_0,v_1,\ldots,v_l\}$ and $g\in PL_\delta(X;\mathcal{K}_0)$ such that $d(g(v_i),v_i)<C$, for some $C>0$ and for all vertices $v_i\in\mathcal{K}_1$. Then there exist $\lambda_i\geq 0$, for $i=1,2,\ldots,l$ such that $x=\displaystyle\sum_{i=1}^{l}\lambda_iv_i$ and $\displaystyle\sum_{i=1}^{l}\lambda_i=1$. Then we can write $g(x)=\displaystyle\sum_{i=1}^{l}\lambda_ig(v_i)$.\\
	Therefore,
	\begin{align*}
		d(g(x),x)&=d\Big(\sum_{i=1}^{l}\lambda_ig(v_i),x\Big)\\&=d\Big(\sum_{i=1}^{l}\lambda_ig(v_i),\sum_{i=1}^{l}\lambda_iv_i\Big)\\&\leq \sum_{i=1}^{l}\lambda_id(g(v_i),v_i)\\&\leq \sum_{i=1}^{l}\lambda_iC=C.
	\end{align*}
Hence our claim is established and this completes the proof.\hspace{3.6cm}\qedsymbol{}\\

Now, we give an example of a path connected l.f.s.c. space $X$ and a simplicial isomorphism $f:X\rightarrow X$ (where both domain and codomain are equipped with the same l.f.s.c. structure) that satisfies the condition (i) of the Corollary \ref{3.6} but not the condition (ii) and also $|f|$ is not a quasi-isometry.
\begin{Ex}
	Let $K\subset \mathbb{R}^{3}$ be a path connected l.f.s.c. consisting of the simplices $\Delta^{n}$ whose vertices are: $v_0^{n}=(n,0,0), ~v_1^{n}=(n+1,0,0), ~v_2^{n}=(n,1,0),~v_3^{n}=(n+\frac {1}{\sqrt{2}},\frac {1}{\sqrt{2}},\frac {1}{n})$ and $K'\subset \mathbb{R}^{3}$ be another path connected l.f.s.c. which consist of the simplices $\Delta^{'n}$ whose vertices are: $u_0^{n}=(-n,0,0),~u_1^{n}=(-n-1,0,0),~u_2^{n}=(-n,-1,0),~u_3^{n}=(-n+\frac {1}{\sqrt{2}},\frac {1}{\sqrt{2}},\frac {1}{n})$, for $n\in\mathbb{N}$. Let $L$ be the 1 simplex from $(-1,0,0)$ to $(1,0,0)$. Let $X$ be the path connected l.f.s.c. defined by $X=K\cup K'\cup L$.\\
	Now, we define the simplicial isomorphism $f:X\rightarrow X$ by $v_i^{n}\mapsto u_i^{n}$, $u_i^{n}\mapsto v_i^{n}$; for all $n\in\mathbb{N}$ and for $i=0,1,2,3$. Clearly we see that for each pair of vertices, there exists a real number $\mu>1$ (here we can take 10) such that $\displaystyle\frac {1}{\mu}\leq \displaystyle\frac {d(f(v_i^{n}),f(v_j^{n}))}{d(v_i^{n},v_j^{n})}\leq \mu$. Again it can be observed that the set $A$, that is the set of cosine of angles of two adjacent edges of $\Delta^{n}$ (and of $\Delta_n^{'}$) is not closed for any order of vertices of $\Delta^{n}$ (and of $\Delta_n^{'}$), since $-\frac {1}{\sqrt{2}}$ is a component of a limit point of $A$ but $-\frac {1}{\sqrt{2}}$ is not appearing in $A$.\\
	Again for two points $x=(n+\frac {1}{2},\frac {1}{2},0)$, $y=(n+\frac {1}{2},\frac {1}{2},\frac {1}{\sqrt{2}n})$ in $\Delta^{n}$, we have\\
	\begin{equation*}
		\frac {d(f(x),f(y))}{d(x,y)}=\frac {d((-n-\frac {1}{2},-\frac {1}{2},0),(-n+\frac {1}{2},\frac {1}{2},\frac {1}{\sqrt{2}n}))}{d((n+\frac {1}{2},\frac {1}{2},0),(n+\frac {1}{2},\frac {1}{2},\frac {1}{\sqrt{2}n}))}=\frac {\sqrt{(2+\frac {1}{2n^2})}}{\frac {1}{2n^2}}\rightarrow\infty~\text{as}~n\rightarrow\infty.
	\end{equation*}
This shows that the map $f$ is not a quasi isometry from $X$ to $X$.
\end{Ex}
\begin{Rem}
	Based on Proof of Theorem \ref{1.1}, we observe that a necessary condition for a function $g$ to belong to the set   $Ker\{PL_\delta(X;\mathcal{K}_0)\rightarrow QI(X)\}$ is that $|d(g(v_i),g(v_j))-d(v_i,v_j)|\leq 2C$, for a pair of vertices $v_i,v_j$ of any simplex of $\mathcal{K}_1$. However, it is important to note that the condition is not sufficient. To illustrate, consider the following example:\\  Let us define a sequence $u_n=n+u_{n-1}$, where $u_0=0$. Then, the function $g:\mathbb{R}\rightarrow\mathbb{R}$ is defined linearly in each case as follows:
	\begin{equation*}
		[-n,-n+1]\mapsto[-n+1,-n+2],~\text{for~all}~n\in\mathbb{N}~\text{and}
	\end{equation*}
\begin{equation*}
	[u_n,u_{n+1}]\mapsto[u_{n+1},u_{n+2}],~\text{for~all~}n\in\mathbb{N}\cup\{0\}.
\end{equation*}
It is clear that $g\in PL_\delta(\mathbb{R})$. Again we see that in the first case, $|d(g(v_i),g(v_j))-d(v_i,v_j)|=0$, and in the second case,  $|d(g(v_i),g(v_j))-d(v_i,v_j)|=1$, based on the defined triangulation of $\mathbb{R}$, and thus, the condition is satisfied. However, it is important to acknowledge that according to the definition of $g$ for $\mathbb{R}_{\geq 0}$, $g$ fails to be quasi-isometrically equivalent to the identity map.
\end{Rem}
	\begin{Rem}\label{3.8}
	The notion of $PL_\delta$-homeomorphism as introduced in Definition \ref{PL_delta} generalizes the concept of $PL_\delta(\mathbb{R})$ defined in \cite{Sankaran}. In the same paper, the author showed that the map from $PL_\delta(\mathbb{R})$ to $QI(\mathbb{R})$ is onto. A similar question arises for $\mathbb{R}^n$ whether the map defined in Theorem \ref{1.1} is surjective. If not, one can further investigate whether the image constitutes a normal subgroup and analyze its quotient.
	\end{Rem}
\section{Center of the group $QI(\mathbb{R}^n)$}\label{4}
	In this section, we prove Theorem \ref{1.2}.
	\subsection{Proof of Theorem \ref{1.2}}
	Let us fix an arbitrary quasi-isometry $f$ of $\mathbb{R}^n$ such that $[f]\neq [id]$. Thus for any $M>0$, there exists $x\in \mathbb{R}^n$ such that $d(f(x),x)\nless M$. Then  we get a sequence $\{a_m\}$ in $\mathbb{R}^n$ where $||a_m||,||f(a_m)||,d(f(a_m),a_m)$ are monotonically increasing sequences of real numbers converging to $\infty$ satisfying $||a_{m+1}||>||f(a_m)||$. \\
	Now we want to obtain a subsequence $\{b_m\}$ of $\{a_m\}$ and a ball $\overline{D_m}$ around $\{f(b_m)\}$ with radius $r_m$ such that\\
	(i) $\{r_m\}$ is a strictly increasing sequence converging to $\infty$,\\
	(ii) the ball $\overline{D_m}$ does not contain any $b_j$ and any $\{f(b_{i})\}$ except for $i=m$, and\\
	(iii) $\overline{D_{i^{'}}}\cap \overline{D_{j^{'}}}=\phi$ for $i'\neq j'$.\\
	Suppose we have got up to $a_{s_m}$ so that a ball $\overline{D_i}$ of radius $r_i$ around $f(a_{s_i})$ does not contain any $a_{s_j}$ for $j\leq i$ and $f(a_{s_j})$ for $j<i$ where $r_i=$ min$\{i,\frac {1}{2}|a_{s_i}-f(a_{s_i})|\}$. Since $a_{s_i}$ and $f(a_{s_i})$ for $i\leq m$ are in a finite ball around the origin, we can choose $a_{s_{m+1}}$ so that a ball $D^{'}_{m+1}$ of radius ${m+1}$ around $f(a_{s_{m+1}})$ does not contain $a_{s_i}$ and $f(a_{s_i})$ for $i\leq m$ and does not intersect $\overline{D_i}$ for $i\leq m$. Then we choose $r_{m+1}$:= min$\{m+1,\frac {1}{2}|f(a_{s_{m+1}})-a_{s_{m+1}}|\}$ and denote the ball with center $f(a_{s_{m+1}})$ and radius $r_{m+1}$ by $\overline{D_{m+1}}$. Then we rename the subsequence $\{a_{s_m}\}$ as $\{b_{m}\}$. \\
	We want to  construct a quasi-isometry $g$ of $\mathbb{R}^{n}$ which will be non identity on each of the ball $\overline{D_m}$ and identity on the complement such that $g$ does not commute with $f$. To do this, we will construct an l.f.s.c. space inside each of the $\overline{D_m}$ and a piecewise-linear homeomorphism of it. \\
	We denote the closed unit ball around the origin by $\overline{\mathbb{D}^1_n}$. We describe
	two simplicial complexes $K$ and $K'$ embedded in $\overline{\mathbb{D}^1_n}$ by specifying its vertices and maximal simplices as follows:\\
	$x_0=(0,0,0,...,1), x_i=(\ubrace{-1,-1,-1,...,-1}{i-1},1,0,0,...,0,-1)/\sqrt{i+1},~1\leq i\leq n-1,\\~x_n=(-1,-1,...,-1,-1)/\sqrt{n}$, $O=(0,0,\ldots,0),~O'=(0,0,\ldots,\frac {1}{4})$.\\
	The simplicial complex $K$ (and $K'$) are defined by subdivisions of the simplex $\{x_0,x_1,\cdots,x_n\}$ after introducing a new vertex $O$ (and $O'$). The maximal simplices of $K$ (and $K'$) are given below.\\
	Let $\Delta_i$ be the simplex with vertices $x_0,x_1,\ldots,x_{i-1},O,x_{i+1},\ldots,x_n$ and $\Delta_i^{'}$ be the simplex with vertices $x_0,x_1,\ldots,x_{i-1},O',x_{i+1},\ldots,x_n$.
	Let $K$ be the simplicial complex with vertices $\{x_0,x_1,\ldots,x_n,O\}$ and maximal simplices $\{\Delta_0,\Delta_1,\ldots,\Delta_n\}$ and $K'$ be the simplicial complex with vertices $\{x_0,x_1,\ldots,x_n,O'\}$ and maximal simplices $\{\Delta_0',\Delta_1',\ldots,\Delta_n'\}$.\\
	Let $h':|K|\rightarrow |K'|$ be the homeomorphism induced by a simplicial map $h':K\rightarrow K'$ defined by $h'(x_i)=x_i,h'(O)=O'$.\\
	Since $K$ and $K'$ are equivalent simplicial complexes of $|K|=|K'|$ and both $K$ and $K'$ consist of only finitely many simplices, so the hypothesis of Corollary \ref{3.6} is satisfied. Therefore, $h':|K|\rightarrow |K'|$ is a quasi-isometry. In fact, there is a constant $k>1$ such that 
	\begin{equation}\label{p1 eq14}
		\frac {1}{k}d(x,y)\leq d(h'(x),h'(y))\leq kd(x,y),~\text{for~all}~ x,y\in |K|.
	\end{equation}
	Since $h'$ is identity on the outer boundary of $|K|$, which consist of the simplices not containing $O$ as a vertex, we can extend $h'$ to a homeomorphism $h:\overline{\mathbb{D}^{1}_n}\rightarrow \overline{\mathbb{D}^1_n}$ defined by 
	\begin{align*}
		h(x)&=h'(x),~\text{for~all~} x\in |K|,\\
		&=x,~\text{for~all~} x\in \overline{\mathbb{D}^1_n}\setminus |K|.
	\end{align*}
	Let $x\in \overline{\mathbb{D}^1_n}\setminus |K|$ and $y\in |K|$, there exist $z\in |K|$ such that the line segments $[x,z]\subset \overline{\mathbb{D}^1_n}\setminus$int$(|K|)$ and the line segment $[z,y]\subset |K|$, so we have
	\begin{align*}
		d(h(x),h(y))&\leq d(h(x),h(z))+d(h(z),h(y))\\
		&\leq k(d(x,z)+d(z,y))\\
		&=kd(x,y).
	\end{align*}
	Thus we get, 
	\begin{equation}\label{p1 eq15}
		\frac {d(h(x),h(y))}{d(x,y)}\leq k, ~\text{for all}~x,y\in \overline{\mathbb{D}^1_n}\setminus |K|.
	\end{equation}
	Similarly, one can easily show,
	\begin{equation}\label{p1 eq16}
		\frac{1}{k}\leq \frac {d(h(x),fh(y))}{d(x,y)},~\text{for all}~x,y\in \overline{\mathbb{D}^1_n}\setminus |K|.
	\end{equation}
	Combining (\ref{p1 eq14}), (\ref{p1 eq15}) and (\ref{p1 eq16}), we obtain that, $h:\overline{\mathbb{D}^1_n}\rightarrow \overline{\mathbb{D}^{1}_n}$ is a quasi-isometry. In fact,
	\begin{align*}
		\displaystyle\frac {1}{k}d(x,y)\leq d(h(x),h(y))\leq kd(x,y)~\text{for~ all }~x,y\in \overline{\mathbb{D}^{1}_n}.
	\end{align*}
	Recall the subsequence $\{b_m\}$ and the ball $\overline{D_m}$ as chosen before.\\
	Now we define $\rho_m:\overline{D_m}\rightarrow \overline{\mathbb{D}^{1}_n}$ by
	\begin{equation*}
		\rho_m(x)=\frac {x-f(b_m)}{r_m},~x\in\overline{D_m}.
	\end{equation*}
	Then we define $g:\mathbb{R}^n\rightarrow \mathbb{R}^n$ by
	\begin{align*}
		g(x)&=(\rho^{-1}_m\circ h\circ \rho_m)(x),~x\in\overline{D_m}\\
		&=x,~x\in \bigcup(\overline{\mathbb{D}_m})^c.
	\end{align*}
	The construction of $g$ has been motivated by ${\Psi(g)}$ in \cite[\textsection 3.2]{Mitra Sankaran}. The proof of the fact that $[g]\in QI(\mathbb{R}^n)$ is similar to the proof of $(ii)$ of Lemma 3.2 in \cite{Mitra Sankaran}.\\ 
	Again we have, 
	\begin{align*}
		g\circ f(b_m)&=(\rho^{-1}_m\circ h\circ \rho_m)(f(b_m))\\
		&=\rho^{-1}_m\circ h(0,0,0,...,0)\\
		&=\rho^{-1}_m(0,0,0,...,\frac {1}{4})\\
		&=f(b_m)+(0,0,0,...,\frac {r_m}{4}) 
	\end{align*}
	and $f\circ g(b_m)=f(b_m).$ So,
	\begin{align*}
		d(f\circ g(b_m),g\circ f(b_m))&=||f\circ g(b_m)-g\circ f(b_m)||\\&=||(0,0,0,\ldots,\displaystyle\frac {r_m}{4})||\rightarrow \infty ~\text{as}~ m\rightarrow \infty.
	\end{align*}
	Hence, $[f]\circ [g]\neq [g]\circ [f]$.\\
	This completes the proof.\hspace{9.8cm}\qedsymbol{}
	
\end{document}